\documentclass[10pt, a4paper]{article}
\usepackage{amsfonts}
\usepackage{amsbsy}
\usepackage{amssymb}
\usepackage{amsmath}
\usepackage{amsmath,amscd}
\usepackage{amsthm}
\usepackage{fancyhdr}
\theoremstyle{plain}
\newtheorem{theorem}{Theorem}[section]

\newtheorem{corollary}{Corollary}[section]
\newtheorem{proposition}{Proposition}[section]
\theoremstyle{definition}

\newtheorem{remark}{Remark}[section]
\newtheorem{example}{Example}[section]

\begin{document}

\title{An $1$--differentiable cohomology induced by a vector field}
\author{Mircea Crasmareanu, Cristian Ida and Paul Popescu}
\date{}
\maketitle
\begin{abstract}
A new cohomology, induced by a vector field, is defined on pairs of differential forms ($1$--differentiable forms) in a manifold. It is proved a link with the classical de Rham cohomology and an $1$-differentable cohomology of Lichnerowicz type associated to an one form. Also, the case when the manifold is complex and the vector field is holomorphic is studied. Finally, an application of this theory to the harmonicity of $1$-differentiable forms is studied in a particular case. 
\end{abstract}
\medskip 
\begin{flushleft}
\strut \textbf{2000 Mathematics Subject Classification:} 14F40; 57R99; 58A10; 58A12.

\end{flushleft}

\section{Introduction}
\setcounter{equation}{0}
Let us consider the field $\Omega^0(M)=\mathcal{F}(M)$ of smooth real valued functions defined on a smooth manifold $M$. For each $p=1,\ldots,n=\dim M$ denote by $\Omega^{p}(M)$ the module of $p$-forms on $M$ and by $\Omega(M)=\oplus_{p \geq 0}\Omega^{p}(M)$ the exterior algebra of $M$. The cohomology of some pair  differential forms (or tensor fields) called  \textit{$1$-differentiable $p$-cochain $\widetilde{\varphi}=(\varphi,\psi)$}, where $\varphi$ is a $p$-form and $\psi$ is a $(p-1)$-form was initiated and intensively studied by Lichnerowicz in \cite{L1} in the context of symplectic and contact geometry and in \cite{L} in the context of Poisson geometry. Further signifiant developments of a such cohomology in the context of Lichnerowicz-Jacobi cohomology can be found, see for instance \cite{L-L-M1, L-L-M2, L-L-M}. Also, we notice that an harmonic and $C$--harmonic theory of $1$-differentiable forms on Sasakian manifolds was recently studied in \cite{I-M}. Another approach concerning calculus on $1$-differentiable forms (called also generalized forms) is proposed in \cite{N-R1, N-R2, R} where some interesting applications in mechanical and physical  field theories are ilustrated. 

On the other hand, another cohomology of some similar pair forms is so called \textit{relative cohomology} with respect to a smooth map $f:M^{\prime}\rightarrow M$ between two smooth manifolds, which is defined in \cite[p. 78]{B-T} as follows: if we consider the set $\Omega^p(f)=\Omega^p(M)\oplus\Omega^{p-1}(M^{\prime})$ and the operator
\begin{displaymath}
d_{f}:\Omega^p(f)\rightarrow\Omega^{p+1}(f)\,,\,d_f(\varphi,\psi)=(d\varphi,f^*\varphi-d\psi),
\end{displaymath}
then $d_f^2=(0,0)$. The cohomology of the differential complex $(\Omega^{\bullet}(f),d_f)$ is called the \textit{relative cohomology with respect to $f$}. We notice that the identity $d_f^2=(0,0)$ is based on the identity $[f^*,d]=0$, where we have denoted by $d$ the exterior derivative both on $M$ and $M^{\prime}$ and $[\cdot,\cdot]$ denotes the graded commutator of two operators which is defined by $[D_1,D_2]=D_1D_2-(-1)^{\deg D_1\cdot\deg D_2}D_2D_1$, see for instance \cite{Mich1}. This fact suggest that we can consider cohomology for a smooth manifold $M$  with respect to the differential complex $(\widetilde{\Omega}^{\bullet}(M),\widetilde{d})$, where $\widetilde{\Omega}^p(M)=\Omega^p(M)\oplus\Omega^{p-1}(M)$ and
\begin{displaymath}
\widetilde{d}:\widetilde{\Omega}^p(M)\rightarrow\widetilde{\Omega}^{p+1}(M)\,,\,\widetilde{d}(\varphi,\psi)=(D_1\varphi,D_2\varphi-D_1\psi),
\end{displaymath}
where $D_1$ is an operator of degree $1$ which satisfies $D_1^2=0$ and $D_2$ is an operator of degree $0$ such that $[D_1,D_2]=0$. 

Taking into account the remarkable identity $[\mathcal{L}_X,d]=0$, where $\mathcal{L}_X$ denotes the Lie derivative with respect to a vector field $X\in\mathcal{X}(M)$, then we can consider $D_1=d$ and $D_2=\mathcal{L}_X$ in the above definition of $\widetilde{d}$ and the resulting cohomology is named by us the \textit{$1$-differentiable cohomology induced by the vector field $X$}.

The aim of this paper is to study this new cohomology as well as some of its possible applications to harmonic theory for $1$-differentiable forms. The paper is organized as follows.

In the second section we define  some differential operators associated to a vector field $X$ on $M$ which act on the exterior graded algebra $(\widetilde{\Omega}^{\bullet}(M),\widetilde{\wedge})$, where $\widetilde{\wedge}$ is the exterior product of $1$-differentiable forms $\widetilde{\varphi}=(\varphi,\psi)$ defined in \cite{L1}. These operators, which are denoted by $\widetilde{d}_X$, $\widetilde{\imath}_X$ and $\widetilde{\mathcal{L}}_X$, represent a natural extension of the operators $d$, $\imath_X$ and $\mathcal{L}_X$, respectively on $\widetilde{\Omega}^{\bullet}(M)$. We prove some natural properties of these operators in relation with the classical properties of $d$, $\imath_X$ and $\mathcal{L}_X$, respectively. Next a cohomology induced by the vector field $X$ denoted by $\widetilde{H}_X^{\bullet}(M)$ is defined as the cohomology of the differential complex $(\widetilde{\Omega}^{\bullet},\widetilde{d}_X)$ and we prove an isomorphism of $\widetilde{H}_X^{\bullet}(M)$ with $H_{dR}^{\bullet}(M)\oplus H_{dR}^{\bullet-1}(M)$, where $H_{dR}^{\bullet}(M)$ denotes de Rham cohomology of $M$. In this way we recover an isomorphism of our cohomology with a classical $1$-differentiable cohomology of Lichnerowicz type associated to an one form. Also, some examples on symplectic manifolds are considered with respect to a symplectic, hamiltonian or Liouville vector field and a Hamilton type equation for $1$-differentiable forms is described. Next, when the manifold is complex and the vector field is holomorphic a similar approach is considered and when the manifold is compact K\"{a}hlerian we established an isomorphism with  $H^{p,\bullet}(M)\oplus H^{p,\bullet-1}(M)$, where $H^{p,\bullet}(M)$ denotes the $p$-th Dolbeault cohomology of $M$. In the third section we consider a generalization of this cohomology with respect to a smooth map $f:M^\prime\rightarrow M$ obtaining a \textit{relative cohomology induced by the vector field $X$}. Then some similar properties as for $\widetilde{d}_X$-cohomology are investigated in this more general case.

In the last section we discuss some aspects concerning harmonicity of $1$-differentiable forms with respect to our operator associated to a Killing vector field in a particular case when $M$ is a $n$-dimensional compact smooth manifold endowed with  a closed $1$-form $\omega$ and with a Riemannian metric $g$ such that $\omega$ is parallel with respect to $g$.

\section{$1$-differentiable cohomology induced by a vector field}
\setcounter{equation}{0}
\subsection{Calculus on $1$-differentiable forms associated to a vector field}
Let us consider  the set $\widetilde{\Omega}^p(M)=\Omega^p(M)\oplus\Omega^{p-1}(M)$. As in formula (5.5) from \cite{L}, we can define a {\it pair wedge product} $\widetilde{\wedge }:\widetilde{\Omega }^p(M)\times \widetilde{\Omega }^{p^{\prime}}(M)\rightarrow \widetilde{\Omega }^{p+p^{\prime}}(M)$ by
\begin{equation}
(\varphi , \psi )\widetilde{\wedge }(\varphi ^{\prime}, \psi ^{\prime})=(\varphi \wedge\varphi ^{\prime},(-1)^p\varphi \wedge\psi ^{\prime}+\psi \wedge\varphi ^{\prime})
\label{II2}
\end{equation}
to be the exterior product on the space $\widetilde{\Omega}^{\bullet}(M)$, where $(\varphi ,\psi )\in \widetilde{\Omega }^p(M),\,(\varphi ^{\prime},\psi ^{\prime})\in \widetilde{\Omega }^{p^{\prime}}(M)$. By this definition, we notice that for an $1$-differentiable $0$-form $(f,0)$, where $f$ is  a smooth function on $M$ we have $(f,0)\cdot(\varphi,\psi)=(f\varphi,f\psi)$. Also, one easily verifies that
\begin{displaymath}
(\varphi,\psi)\widetilde{\wedge}\left((\varphi^{\prime},\psi^{\prime})+(\varphi^{\prime\prime},\psi^{\prime\prime})\right)=(\varphi,\psi)\widetilde{\wedge}(\varphi^{\prime},\psi^{\prime})+(\varphi,\psi)\widetilde{\wedge}(\varphi^{\prime\prime},\psi^{\prime\prime}),
\end{displaymath}
\begin{displaymath}
(\varphi,\psi)\widetilde{\wedge}\left((\varphi^{\prime},\psi^{\prime})\widetilde{\wedge}(\varphi^{\prime\prime},\psi^{\prime\prime})\right)=\left((\varphi,\psi)\widetilde{\wedge}(\varphi^{\prime},\psi^{\prime})\right)\widetilde{\wedge}(\varphi^{\prime\prime},\psi^{\prime\prime})
\end{displaymath}
and
\begin{displaymath}
(\varphi,\psi)\widetilde{\wedge}(\varphi^{\prime},\psi^{\prime})=(-1)^{p p^{\prime}}(\varphi^{\prime},\psi^{\prime})\widetilde{\wedge}(\varphi,\psi),
\end{displaymath}
which say that $(\widetilde{\Omega}^{\bullet}(M),\widetilde{\wedge})$ is a graded algebra.

Now we consider $X\in\mathcal{X}(M)$ a fixed vector field on $M$ and we define the following differential operators acting on $\widetilde{\Omega}^{\bullet}(M)$:
\begin{equation}
\widetilde{d}_X:\widetilde{\Omega }^p(M)\rightarrow \widetilde{\Omega }^{p+1}(M), \quad \widetilde{d}_X(\varphi , \psi )=(d\varphi , \mathcal{L}_X\varphi -d\psi ).
\label{II1}
\end{equation}
\begin{equation}
\widetilde{\imath}_X:\widetilde{\Omega }^p(M)\rightarrow \widetilde{\Omega }^{p-1}(M)\,,\,\widetilde{\imath}_X(\varphi , \psi )=(\imath_X\varphi , -\imath_X\psi )
\label{II9}
\end{equation}
and
\begin{equation}
\widetilde{\mathcal{L}}_X:\widetilde{\Omega }^p(M)\rightarrow\widetilde{\Omega }^{p}(M)\,,\,\widetilde{\mathcal{L}}_X(\varphi , \psi )=(\mathcal{L}_X\varphi , \mathcal{L}_X\psi ).
\label{II10}
\end{equation}
By direct calculus using \eqref{II2}-\eqref{II10} and the classical relations concerning the operators $d$, $i_X$ and $\mathcal{L}_X$ the following relations hold:
\begin{equation}
\widetilde{d}_X((\varphi , \psi )\widetilde{\wedge}(\varphi ^{\prime}, \psi ^{\prime}))=\widetilde{d}_X(\varphi , \psi )\widetilde{\wedge}(\varphi ^{\prime}, \psi ^{\prime})+(-1)^p(\varphi , \psi )\widetilde{\wedge}\widetilde{d}_X(\varphi ^{\prime}, \psi ^{\prime}),
\label{II3}
\end{equation}
\begin{equation}
\widetilde{\imath}_X((\varphi , \psi )\widetilde{\wedge}(\varphi ^{\prime}, \psi ^{\prime}))=\widetilde{\imath}_X(\varphi , \psi )\widetilde{\wedge}(\varphi ^{\prime}, \psi ^{\prime})+(-1)^p(\varphi , \psi )\widetilde{\wedge}\,\widetilde{\imath}_X(\varphi ^{\prime}, \psi ^{\prime}),
\label{II11}
\end{equation}
and
\begin{equation}
\widetilde{\mathcal{L}}_X((\varphi , \psi )\widetilde{\wedge}(\varphi ^{\prime}, \psi ^{\prime}))=\widetilde{\mathcal{L}}_X(\varphi , \psi )\widetilde{\wedge}(\varphi ^{\prime}, \psi ^{\prime})+(\varphi , \psi )\widetilde{\wedge}\widetilde{\mathcal{L}}_X(\varphi ^{\prime}, \psi ^{\prime}),
\label{IIx11}
\end{equation}
which say that $\widetilde{d}_X$ is an antiderivation of degree $1$, $\widetilde{\imath}_X$ is an antiderivation of degree $-1$ and $\widetilde{\mathcal{L}}_X$ is a derivation of degree $0$, respectively, in the exterior graded algebra $(\widetilde{\Omega }^p(M), \widetilde{\wedge})$.

Using the above definitions and the classical properties concerning the operators $d$, $\imath_X$ and $\mathcal{L}_X$ we easily obtain:
\begin{equation}
\widetilde{d}_X^2=(0,0)\,,\,\widetilde{\imath}_X^2=(0,0)\,,\,[\widetilde{\mathcal{L}}_X, \widetilde{\mathcal{L}}_Y]=\widetilde{\mathcal{L}}_{[X, Y]}\,,\,[\widetilde{\mathcal{L}}_X, \widetilde{\imath}_Y]=\widetilde{\imath}_{[X, Y]}.
\label{II13}
\end{equation}
For instance the first relation from \eqref{II13} it follows by direct calculus in \eqref{II1} using $[\mathcal{L}_X,d]=0$.

Also, we have
\begin{proposition}
If $Y\in \mathcal{X}(M)$ is a Lie symmetry of $X$ i.e. $[X, Y]=0$ then the following remarkable identities holds:
\begin{equation}
\widetilde{\mathcal{L}}_Y=\widetilde{d}_X\widetilde{\imath}_Y+\widetilde{\imath}_Y\widetilde{d}_X\,,\,[\widetilde{\mathcal{L}}_Y, \widetilde{d}_X]=(0,0).
\label{II14}
\end{equation}
\end{proposition}
As a consequence one gets
\begin{equation}
\widetilde{\mathcal{L}}_X=\widetilde{d}_X\widetilde{\imath}_X+\widetilde{\imath}_X\widetilde{d}_X\,,\,[\widetilde{\mathcal{L}}_X, \widetilde{d}_X]=(0,0).
\label{IIx14}
\end{equation}
Let $f:M^{\prime}\rightarrow M$ be a smooth map between two smooth manifolds. We define $\widetilde{f}^*:\widetilde{\Omega}^p(M)\rightarrow\widetilde{\Omega}^p(M^{\prime})$ by
\begin{equation}
\widetilde{f}^*(\varphi,\psi)=(f^*\varphi,f^*\psi)\,,\forall\,(\varphi,\psi)\in\widetilde{\Omega}^p(M).
\label{x01}
\end{equation}
It is well known that if $f:M^{\prime}\rightarrow M$ is a diffeomorphism, then we have 
\begin{equation}
f^*d\varphi=df^*\varphi\,,\,f^*(\imath_{f_*X}\varphi)=\imath_X(f^*\varphi)\,,\,f^*(\mathcal{L}_{f_*X}\varphi )=\mathcal{L}_{X}(f^*\varphi ),
\label{1}
\end{equation}
for any $X\in \mathcal{X}(M^{\prime})\,,\,\varphi \in \Omega ^{\bullet}(M)$, and by direct calculus we obtain
\begin{proposition}
If $f:M^{\prime}\rightarrow M$ is a diffeomorphism the following relations hold:
\begin{equation}
\widetilde{d}_X\widetilde{f}^*=\widetilde{f}^*\widetilde{d}_{f_*X}\,,\,\widetilde{f}^*\widetilde{\imath}_{f_*X}=\widetilde{\imath}_X\widetilde{f}^*\,,\,\widetilde{f}^*\widetilde{\mathcal{L}}_{f_*X}=\widetilde{\mathcal{L}}_{X}\widetilde{f}^*\,,\, X\in \mathcal{X}(M^{\prime}).
\label{01}
\end{equation}
\end{proposition}

\subsection{$\widetilde{d}_X$-cohomology}

As well as we seen by the first relation of \eqref{II13} we have a differential complex $(\widetilde{\Omega }^{\bullet}(M), \widetilde{d}_X)$ and denote by  $\widetilde{H}_X^{\bullet}(M)$ the cohomology of this complex. We will call it the \textit{$1$-differentiable cohomology induced by the vector field} $X$ of the smooth manifold $M$ or $\widetilde{d}_X$-cohomology of $M$. 

We have
\begin{proposition}
The exterior product from \eqref{II2} induces, on the space of $\widetilde{d}_X$-cohomology classes, a cohomology algebra structure.
\end{proposition}

If $f:M^{\prime}\rightarrow M$ is a diffeomorphism between two smooth manifolds then the first relation \eqref{01} says that the map $\widetilde{f}^*$ from \eqref{x01} induces the following map in the cohomology level:
\begin{equation}
\widetilde{f}^{\bullet}:\widetilde{H}_{f_*X}^{\bullet}(M)\rightarrow \widetilde{H}_{X}^{\bullet}(M^{\prime})\,,\,\widetilde{f}^{\bullet}[(\varphi , \psi )]=[\widetilde{f}^*(\varphi , \psi )]\,,\,\forall\,X\in\mathcal{X}(M^{\prime}).
\label{2}
\end{equation}

We notice that a cohomology class in $\widetilde{H}^{\bullet}_X(M)$ is represented by a closed form $\varphi $ whose Lie derivative with respect to $X$ is an exact form. Because for any closed form $\varphi$ we have $\mathcal{L}_X\varphi=d\imath_X\varphi$ (that is $\mathcal{L}_X\varphi$ is an exact form), then it is easy to see that we have the natural maps in cohomology:
\begin{displaymath}
H_{dR}^{\bullet}(M)\mapsto \widetilde{H}_X^{\bullet}(M)\,,\,[\varphi ]\mapsto [(\varphi , \imath_X\varphi )],
\end{displaymath}
\begin{displaymath}
\widetilde{H}_X^{\bullet}(M)\mapsto H_{dR}^{\bullet-1}(M)\,,\,[(\varphi,\psi) ]\mapsto [\imath_X\varphi-\psi].
\end{displaymath}
Moreover the map
\begin{displaymath}
\widetilde{H}^{\bullet}_X(M)\mapsto H^{\bullet}_{dR}(M)\oplus H^{\bullet-1}_{dR}(M)\,,\,[(\varphi,\psi)]\mapsto  ([\varphi],[\imath_X\varphi-\psi])
\end{displaymath}
with the inverse
\begin{displaymath}
H^{\bullet}_{dR}(M)\oplus H^{\bullet-1}_{dR}(M)\mapsto \widetilde{H}^{\bullet}_X(M)\,,\,([\varphi],[\psi])\mapsto[(\varphi,\imath_X\varphi-\psi)]
\end{displaymath}
establishes the following isomorphism:
\begin{equation}
\widetilde{H}^{\bullet}_X(M)\cong H^{\bullet}_{dR}(M)\oplus H^{\bullet-1}_{dR}(M).
\label{IIx}
\end{equation}
\begin{remark}
This cohomology is isomorphic with an $1$-differentiable cohomology associated to an $1$-form of Lichnerowicz type. More exactly, for any  $\eta\in\Omega^1(M)$ we consider the operator defined by
\begin{equation}
\widetilde{d}_{\eta}:\widetilde{\Omega}^p(M)\rightarrow\widetilde{\Omega}^{p+1}(M)\,\,\widetilde{d}_{\eta}(\varphi,\psi)=(d\varphi-d\eta\wedge\psi,-d\psi).
\label{f1}
\end{equation}
By direct calculus we gets $\widetilde{d}_{\eta}^2=(0,0)$ and we have a differential complex $(\widetilde{\Omega }^{\bullet}(M), \widetilde{d}_{\eta})$. Denote by  $\widetilde{H}_{\eta}^{\bullet}(M)$ the cohomology of this complex and we will call it the \textit{$1$-differentiable cohomology associated to the $1$-form $\eta$} of the smooth manifold $M$ or $\widetilde{d}_{\eta}$-cohomology of $M$. This cohomology is a particular case of a more general \textit{$1$-differentiable cohomology} firstly defined by Lichnerowicz in \cite{L1} in the context of symplectic and contact manifolds and \cite{L} in the context of Poisson manifolds, respectively.  Also, according to \cite{L-L-M2,L1,L} we have $\widetilde{H}_{\eta}^{\bullet}(M)\cong H^{\bullet}_{dR}(M)\oplus H^{\bullet-1}_{dR}(M)$ and the isomophism \eqref{IIx} says that $\widetilde{H}_{X}^{\bullet}(M)\cong\widetilde{H}_{\eta}^{\bullet}(M)$.
\end{remark}

Consequently, $\dim\widetilde{H}_X^p(M)=b_p(M)+b_{p-1}(M)$, where $b_p(M)$ is the $p$-th Betti number of $M$. In particular, the dimension of $\widetilde{H}_X^p(M)$ is a topological invariant of $M$, for all $p$. Also, by applying the Poincar\'{e} duality for the de Rham cohomology $H^{\bullet}_{dR}(M)$ in \eqref{IIx} we obtain the following Poincar\'{e} duality for $\widetilde{d}_X$-cohomology:
\begin{equation}
\widetilde{H}_X^p(M)\cong \left(\widetilde{H}^{n+1-p}_{c,X}(M)\right)^*,\,p=0,\ldots,n+1,
\label{k1}
\end{equation}
where the index $c$ denotes \textit{compactly supported}.

Indeed, taking into account the Poincar\'{e} duality $H_{dR}^p(M)\cong \left(H_{c,dR}^{n-p}(M)\right)^*$ and the classical isomorphism $(V\oplus W)^*\cong V^*\oplus W^*$ then we easily obtain \eqref{k1}.

Also, we have
\begin{proposition}
$\widetilde{H}_X^p(M)=\widetilde{0}$ for any $p>n+1$.
\end{proposition}
\begin{proposition}
For any vector fields $X,Y\in\mathcal{X}(M)$ there is an isomorphism:
\begin{equation}
\label{kI1}
\alpha_{X,Y}:\widetilde{H}^{\bullet}_X(M)\rightarrow\widetilde{H}^{\bullet}_Y(M)\,,\,\alpha_{X,Y}([(\varphi,\psi)])=[(\varphi,\imath_{Y-X}\varphi+\psi)]
\end{equation}
with the inverse $\alpha_{X,Y}^{-1}=\alpha_{Y,X}$.
\end{proposition}

\subsubsection{Examples}

\begin{example}Let $(M,\omega)$ be a symplectic manifold, that is $\omega\in\Omega^2(M)$ is closed and nondegenerated. We recall that a vector field $X\in\mathcal{X}(M,\omega)$ is called \textit{symplectic} if $\imath_X\omega$ is closed and it is called \textit{hamiltonian} if $\imath_X\omega$ is exact. Taking into account that for every symplectic vector field $X\in\mathcal{X}(M,\omega)$ we have $\mathcal{L}_X\omega=0$ it follows that $\widetilde{d}_X(\omega,\theta)=(0,0)$ for every symplectic vector field $X\in\mathcal{X}(M,\omega)$ and every  closed $1$-form $\theta$ on $M$.
\end{example}
Moreover we have the following result:
\begin{proposition}
If $[(\omega,\theta)]=(0,0)$ and $\theta$ is exact then $\omega$ is exact and $X$ is hamiltonian.
\end{proposition} 
\begin{proof}
Let $\theta=dg$, where $g\in C^\infty(M,\omega)$. Then $[(\omega,dg)]=(0,0)$ if there exists $(\varphi,f)\in\widetilde{\Omega}^1(M)$ such that $(\omega,dg)=\widetilde{d}_X(\varphi,f)$ which implies
\begin{displaymath}
\omega=d\varphi\,\,\,{\rm and}\,\,\,\mathcal{L}_X\varphi-df=dg.
\end{displaymath}
The first relation from above says that $\omega$ is exact and replacing $\omega=d\varphi$ in the second relation we obtain $\imath_X\omega=d(f+g-\imath_X\varphi)$, were we have used the Cartan identity $\mathcal{L}_X=d\imath_X+\imath_Xd$. Thus, $\imath_X\omega$ is exact i.e., $X$ is hamiltonian.
\end{proof}
\begin{example}
Let $(M,\omega)$ be a symplectic manifold and $X\in\mathcal{X}(M,\omega)$ be a symplectic vector field. Then $(\imath_X\omega,k)$ is $\widetilde{d}_X$-closed for every $k\in\mathbb{R}$. Moreover if $k\neq0$ then $[(\imath_X\omega,k)]\neq(0,0)$. 

Indeed, if we suppose that $[(\imath_X\omega,k)]=(0,0)$ then there exists $(f,0)\in\widetilde{\Omega}^0(M)$ such that $(\imath_X\omega,k)=\widetilde{d}_X(f,0)$ which implies
\begin{displaymath}
df=\imath_X\omega\,\,\,{\rm and}\,\,\, \mathcal{L}_Xf=k.
\end{displaymath}
The first relation from above says that $X$ is hamiltonian and if $H$ is the hamiltonian function of $X$, that is $\imath_X\omega=-dH$, then we have $d(f+H)=0$ which implies $f=-H+c$, where $c\in\mathbb{R}$. Now, replacing $f=-H+c$ in the second relation and taking into account that $\mathcal{L}_XH=0$ we obtain $\mathcal{L}_Xf=0\neq k$. 
\end{example}
\begin{example}
Suppose now that $(M, \omega )$ is a symplectic manifold with $\omega \in \Omega ^2(M)$ an exact symplectic form i.e. there exists an $1$-form $\theta $ (sometimes called \textit{Liouville $1$-form}) such that: $\omega =d\theta $. Then there exists a {\it potential symplectic vector field} $\xi $ (sometimes called \textit{Liouville vector field}) i.e. a vector field for which: $\theta =\imath_{\xi }\omega $, conform \cite[p. 22]{c:ds}. Then, we have $\widetilde{d}_{\xi }(\omega , \theta )=(0, {\mathcal L}_{\xi }\omega -\omega )$. But ${\mathcal L}_{\xi }\omega =d\imath_{\xi }\omega +\imath_{\xi }d\omega =d\theta =\omega$
and then $\widetilde{d}_{\xi }(\omega , \theta )=(0, 0)$ which means that $[(\omega , \theta )]\in \widetilde{H}_{\xi }^{2}(M)$. The $\widetilde{d}_{\xi}$-cohomology class $[(\omega,\theta)]$ from this example will be called the \textit{$1$-differentiable cohomology class induced by the Liouville vector field $\xi$ of the exact symplectic manifold $(M,\omega)$}.
\end{example}

We have
\begin{proposition}
The $\widetilde{d}_{\xi}$-cohomology class $[(\omega,\theta)]$ induced by the Liouville vector field $\xi$ of the exact symplectic manifold $(M,\omega)$ vanish.
\end{proposition}
\begin{proof}
Let $f\in C^{\infty}(M)$. Taking into acount $d\theta=\omega$, $\mathcal{L}_{\xi}\theta=\theta$ and $\mathcal{L}_{\xi}df=d(\xi f)$, then it follows that $(\omega,\theta)=\widetilde{d}_{\xi}(\theta+df, \xi f)$.
\end{proof}
\begin{remark}
The vanishing of the $\widetilde{d}_{\xi}$-cohomology class $[(\omega,\theta)]$ it follows also from the isomorphism \eqref{IIx}, since $[(\omega,\theta)]\mapsto([\omega], [\imath_{\xi}\omega-\theta])=(0,0)$.
\end{remark}

\begin{example}
(Hamilton type equations.) Let $\widetilde{\omega}=(\omega ,\theta)\in \widetilde{\Omega}^{2}\left( M\right) $ such that $\widetilde{d}_{X}\widetilde{\omega}=(0,0)$ and $\omega \in \Omega ^{2}\left( M\right) $ has a maximal rank. Since $d\omega =0$ and $d\imath_{X}\omega =d\theta $, it follows that $\omega $ is a symplectic form and we can consider (at least locally) $\theta $ having the form $\theta =\imath_{X}\omega +dh$, where $h\in \Omega ^{0}\left( M\right) =C^\infty(M)$. Also, for $\widetilde{\Theta}=\left( f,0\right) \in \widetilde{\Omega}^{0}\left( M\right) $ we have $\widetilde{d}_{X}\widetilde{\Theta}=\left( df,Xf\right) $.
The equation $-\widetilde{d}_{X}\widetilde{\Theta}=\imath_{X}\widetilde{\omega}$, i.e. $-\left(df,Xf\right) =(\imath_{X}\omega ,-\imath_{X}dh)$ reads $\imath_{X}\omega =-df$ and $Xh =X f$.

It follows that given $\widetilde{\omega}=(\omega ,\theta )$ which is $\widetilde{d}_{X}$-closed, $\omega $ of maximal rank, and $\widetilde{\Theta}=(f,0)\in \widetilde{\Omega}^{0}\left( M\right) $, then the Hamilton type equation $-\widetilde{d}_{X}\widetilde{\Theta}=\imath_{X}\widetilde{\omega}$ has as solution the $\omega$-hamiltonian vector
field $X$ corresponding to $f$, provided that $\theta $ is an exact form, i.e. $\theta =dh^{\prime }$, $h^{\prime }\in C^\infty(M)$, and $Xh^{\prime } =0$. In the particular case when $\theta =0$, then $X$ is simply the hamiltonian vector field corresponding to $f$.
\end{example}

\subsection{The complex case}

The main purpose of this subsection is to formulate a cohomology induced by a holomorphic vector field on complex manifolds and possible relation with classical Dolbeault cohomology. 

Let $M$ be a complex $n$-dimensional manifold and $X$  a holomorphic vector field on $M$. Denote by $\Omega^{p,q}(M)$ the space of all complex differential forms on $M$ of type $(p,q)$, briefly $(p,q)$-forms on $M$. Then we have a decomposition of the exterior derivative $d=\partial+\overline{\partial}$, where $\partial:\Omega^{p,q}(M)\rightarrow\Omega^{p+1,q}(M)$ and $\overline{\partial}:\Omega^{p,q}(M)\rightarrow\Omega^{p,q+1}(M)$. Similarly to previous notations we consider the set $\widetilde{\Omega}^{p,q}(M)=\Omega^{p,q}(M)\oplus\Omega^{p,q-1}(M)$. Since for every holomorphic vector field $X$ on $M$ the Lie derivative $\mathcal{L}_X$ preserves the $(p,q)$ type of forms on which this act, we define the following operator:
\begin{equation}
\overline{\partial}_X:\widetilde{\Omega}^{p,q}(M)\rightarrow \widetilde{\Omega}^{p,q+1}(M)\,,\,\overline{\partial}_X(\varphi,\psi)=(\overline{\partial}\varphi, \mathcal{L}_X\varphi-\overline{\partial}\psi),\,\,\forall\,(\varphi,\psi)\in\widetilde{\Omega}^{p,q}(M).
\label{c1}
\end{equation}
Also, taking the same components in relation $\mathcal{L}_X (\partial+\overline{\partial})=(\partial+\overline{\partial})\mathcal{L}_X$ we get $[\mathcal{L}_X,\overline{\partial}]=0$.

It is easy to see that $\overline{\partial}_X^2=(0,0)$. Then for any $p\geq0$ we have the differential complex $(\widetilde{\Omega}^{p,\bullet},\overline{\partial}_X)$ and denote by $\widetilde{H}_X^{p,\bullet}(M)$ the cohomology groups of this complex, called  \textit{$p$-th Dolbeault $1$-differentiable cohomology induced by the holomorphic vector field $X$} of $M$.

We notice that the exterior product from \eqref{II2} can be extended to $\widetilde{\Omega}^{p,q}(M)$ and this induces, on the space of $\overline{\partial}_X$-cohomology classes, a cohomology algebra structure. Also, if we consider a biholomorphic map $f:M^{\prime}\rightarrow M$ between two complex manifolds, taking into account that for any holomorphic vector field $X$ on $M^{\prime}$, $f_*X$ is a holomorphic vector field on $M$ and $f^*$ preserves the $(p,q)$-forms, then the map $\widetilde{f}^*$ from \eqref{x01} induces the following map in the cohomology level:
\begin{equation}
\widetilde{f}^{\bullet}:\widetilde{H}_{f_*X}^{p,\bullet}(M)\rightarrow \widetilde{H}_{X}^{p,\bullet}(M^{\prime})\,,\,\widetilde{f}^{\bullet}[(\varphi , \psi )]=[\widetilde{f}^*(\varphi , \psi )]
\label{c2}
\end{equation}
for any holomorphic vector field $X$ on $M^{\prime}$.

In the following we are interested to find an isomorphism of type \eqref{IIx} for Dolbeault $1$-differentiable cohomology induced by the holomorphic vector field $X$ of $M$ and we have

\begin{theorem}
\label{t2.1}
If $(M,g)$ is a compact K\"{a}hler manifold then we have the following isomorphism:
\begin{equation}
\widetilde{H}^{p,\bullet}_X(M)\cong H^{p,\bullet}(M)\oplus H^{p,\bullet-1}(M),
\label{c6}
\end{equation}
where $H^{p,\bullet}(M)$ denotes the classical $p$-th Dolbeault cohomology of $M$.
\end{theorem}
\begin{proof}
Using the Hodge theory on compact K\"{a}hler manifolds we obtain that any Dolbeault class in $H^{p,q}(M)$ is represented by a form $\varphi\in\Omega^{p,q}(M)$ that is closed both with respect to $d$ and $\overline{\partial}$, see \cite{Ko, Mo-Ko}. Thus, for any Dolbeault class $[\varphi]\in H^{p,q}(M)$ we have $\mathcal{L}_{X}\varphi=d(\imath_{X}\varphi)$ is again a $(p,q)$-form on $M$. Since $X$ is holomorphic we have $\imath_{X}\varphi\in\Omega^{p-1,q}(M)$ and so
\begin{displaymath}
\mathcal{L}_{X}\varphi=\partial(\imath_{X}\varphi)\,\,{\rm and}\,\,\overline{\partial}(\imath_{X}\varphi)=0
\end{displaymath}
for degree reasons. As $\overline{\partial}(\imath_{X}\varphi)=0$  we find a $(p-1,q-1)$-form $\varphi^{'}$ on $M$ such that $\imath_{X}\varphi+\overline{\partial}\varphi^{'}$ is $\overline{\partial}$-closed which implies according to our hypothesis that $\imath_{X}\varphi+\overline{\partial}\varphi^{'}$ is $d$-closed. Hence
\begin{displaymath}
\partial(\imath_{X}\varphi)=-\partial\overline{\partial}\varphi^{'}=\overline{\partial}\partial\varphi^{'}
\end{displaymath}
and so
\begin{equation}
\mathcal{L}_{X}\varphi=\overline{\partial}\partial\varphi^{'},
\label{c7}
\end{equation}
which shows that the induced action on Dolbeault cohomology $H^{p,q}(M)$ is trivial. Then the map
\begin{displaymath}
\widetilde{H}^{p,\bullet}_X(M)\mapsto H^{p,\bullet}(M)\oplus H^{p,\bullet-1}(M)\,,\,[(\varphi,\psi)]\mapsto  ([\varphi],[\partial\varphi^\prime-\psi])
\end{displaymath}
with the inverse
\begin{displaymath}
H^{p,\bullet}(M)\oplus H^{p,\bullet-1}(M)\mapsto \widetilde{H}^{p,\bullet}_X(M)\,,\,([\varphi],[\psi])\mapsto[(\varphi,\partial\varphi^\prime-\psi)]
\end{displaymath}
establishes the isomorphism from \eqref{c6}.
\end{proof}

\begin{remark}
Using the Serre duality for Dolbeault cohomology on compact K\"{a}hler manifolds we obtain the following Serre duality for $\overline{\partial}_X$-cohomology:
\begin{equation}\widetilde{H}_X^{p,q}(M)\cong \widetilde{H}_X^{n-p,n+1-q}(M).
\label{c9}
\end{equation}
\end{remark}

\section{A relative cohomology induced by a vector field}
\setcounter{equation}{0}

In this section we present a generalization of the $\widetilde{d}_X$-cohomology relative to a smooth map $f:M^\prime\rightarrow M$ between two manifolds. 

As in the case of relative cohomology, we consider the set $\Omega^p(f)=\Omega^p(M)\oplus\Omega^{p-1}(M^{\prime})$ and the operator $d_{X,f}$ defined by
\begin{equation}
d_{X,f}:\Omega^p(f)\rightarrow\Omega^{p+1}(f)\,,\,d_{X,f}(\varphi,\psi)=(d\varphi,\mathcal{L}_{X}f^*\varphi-d\psi)\,,\,\,X\in\mathcal{X}(M^\prime).
\label{a1}
\end{equation} 
An easy calculation shows that $d_{X,f}^2=(0,0)$. Denote the cohomology groups of the complex $(\Omega^{\bullet}(f),d_{X,f})$ by $H^{\bullet}_{X}(f)$ and we call it the \textit{relative cohomology induced by the vector field $X$}.
\begin{remark}
If $M^\prime=M$ and $f=Id|_M$ then $H^{\bullet}_{X}(Id|_M)=\widetilde{H}_X^\bullet(M)$.
\end{remark}

As in formula \eqref{II2}, we can define a {\it pair wedge product} adapted to our study by
\begin{equation}
\wedge_f:\Omega ^p(f)\times \Omega ^{p^{\prime}}(f)\rightarrow \Omega ^{p+p^{\prime}}(f)\,,\,\,(\varphi,\psi)\wedge_f(\varphi^{\prime},\psi^{\prime})=(\varphi\wedge\varphi^{\prime},(-1)^pf^*\varphi\wedge\psi^{\prime}+\psi\wedge f^*\varphi^{\prime}),
\label{xy1}
\end{equation}
where $(\varphi,\psi)\in\Omega^p(f)$ and $(\varphi^{\prime},\psi^{\prime})\in\Omega^{p^{\prime}}(f)$, respectively.

One easily verifies that this product is anticomutative, namely
\begin{displaymath}
(\varphi,\psi)\wedge_f(\varphi^{\prime},\psi^{\prime})=(-1)^{p p^{\prime}}(\varphi^{\prime},\psi^{\prime})\wedge_f(\varphi,\psi)
\end{displaymath}
and also $(\Omega^p(f),\wedge_f)$ is an exterior graded algebra. By direct calculus we have that $d_{X,f}$ is an antiderivation with respect to the exterior product from \eqref{xy1}. We have
\begin{proposition}
The exterior product from \eqref{xy1} induces, on the space of $d_{X,f}$-cohomology classes, a cohomology algebra structure.
\end{proposition}

\begin{proposition}
Let $\varphi\in\Omega^p(M)$ be a closed $p$-form on $M$, i.e. $d\varphi=0$. Then for every vector field $X\in\mathcal{X}(M^\prime)$  we have $(\varphi,\imath_Xf^*\varphi)\in\Omega^p(f)$ is $d_{X,f}$-closed and, moreover if $\varphi$ is exact then the $d_{X,f}$-cohomology class $[(\varphi,\imath_Xf^*\varphi)]$  vanishes. 
\end{proposition}
\begin{proof}
Taking into account $d\varphi=0$ and Cartan identity $\mathcal{L}_X=d\imath_X+\imath_Xd$,  we have
\begin{eqnarray*}
d_{X,f}(\varphi,\imath_Xf^*\psi)&=&(d\varphi,\mathcal{L}_{X}f^*\varphi-d\imath_Xf^*\varphi)=(0,d\imath_Xf^*\varphi-d\imath_Xf^*\varphi)=(0,0).
\end{eqnarray*}
Now, if $\varphi=d\psi$, by direct calculus  we have $d_{X,f}(\psi,\imath_Xf^*\psi+d\zeta)=(\varphi,\imath_Xf^*\varphi)$, for every $\zeta\in\Omega^{p-2}(M^\prime)$, which says that $[(\varphi,\imath_Xf^*\varphi)]\in H^p_X(f)$ vanishes.
\end{proof}

\begin{example}
Let $(M, \omega )$ and $(M^\prime,\omega^\prime)$ be two exact symplectic manifolds with $\theta,\xi$ and $\theta^\prime,\xi^\prime$ the corresponding  Liouville $1$-forms and vector fields, respectively. We also consider $f:(M^\prime,\omega^\prime)\rightarrow(M,\omega)$ a symplectomorphism between two exact symplectic manifolds, that is $\omega^\prime=f^*\omega$. Then $d_{\xi^\prime,f }(\omega , \theta^\prime )=(0, {\mathcal L}_{\xi^\prime }f^*\omega -d\theta^\prime)$. But ${\mathcal L}_{\xi^\prime}f^*\omega =\mathcal{L}_{\xi^\prime}\omega^\prime=d\theta^\prime$, which implies $d_{\xi^\prime,f }(\omega , \theta^\prime )=(0, 0)$, that is $[(\omega , \theta^\prime )]\in H_{\xi^\prime }^{2}(f)$, which is called the \textit{cohomology class induced by the Liouville vector field $\xi^\prime$ relative to the symplectomorphism $f:(M^\prime,\omega^\prime)\rightarrow(M,\omega)$ between two exact symplectic manifolds}. Moreover if $f^*\theta=\theta^\prime$ then $[(\omega , \theta^\prime )]$ vanishes.
\end{example}

As in the classical case of relative cohomology, \cite{B-T}, we  can consider the mappings $\alpha:\Omega^{p-1}(M^{\prime})\rightarrow\Omega^p(f)$, $\alpha(\psi)=(0,\psi)$ and $\beta:\Omega^p(f)\rightarrow\Omega^{p}(M)$, $\beta(\varphi,\psi)=\varphi$ for all $\varphi\in\Omega^p(M)$ and $\psi\in\Omega^{p-1}(M^{\prime})$, respectively. Then,  we have the following result which relates the cohomology  $H_{X}^{\bullet}(f)$ with the de Rham cohomology of $M$ and $M^{\prime}$, respectively.
\begin{proposition}
\label{p2.1}
Let $f:M^{\prime}\rightarrow M$ be a smooth map between two manifolds. Then:
\begin{enumerate}
\item[(i)] The mappings $\alpha$ and $\beta$ induce an exact sequence of complexes
\begin{displaymath}
0\longrightarrow(\Omega^{\bullet-1}(M^{\prime}), d)\stackrel{\alpha}{\longrightarrow}(\Omega^{\bullet}(f),d_{X,f})\stackrel{\beta}{\longrightarrow}(\Omega^{\bullet}(M),d)\longrightarrow0.
\end{displaymath}
\item[(ii)] This exact sequence induces a long exact cohomology sequence
\begin{equation}
\ldots\longrightarrow H^{p-1}_{dR}(M^{\prime})\stackrel{\alpha^*}{\longrightarrow}H_{X}^p(f)\stackrel{\beta^*}{\longrightarrow}H^{p}_{dR}(M)\stackrel{\delta^*_{p}}{\longrightarrow}H^{p}_{dR}(M^{\prime})\longrightarrow\ldots,
\label{3}
\end{equation}
where the connecting homomorphism $\delta^*_{p}$ is defined by
\begin{equation}
\delta^*_{p}([\varphi])=[\mathcal{L}_{X}f^*\varphi],\,\,{\rm for\,\,any}\,\,[\varphi]\in H^{p}_{dR}(M).
\label{4}
\end{equation}
\end{enumerate}
\end{proposition}
Since for any closed form $\varphi\in\Omega^p(M)$ we have $\mathcal{L}_{X}f^*\varphi=d(i_Xf^*\varphi)$ the connecting homomorphism $\delta^*_{p}=0$ and so for every $p\geq1$ we have an exact sequence:
\begin{displaymath}
0\longrightarrow H^{p-1}_{dR}(M^{\prime})\stackrel{\alpha^*}{\longrightarrow}H_{X}^p(f)\stackrel{\beta^*}{\longrightarrow}H^{p}_{dR}(M)\stackrel{\delta^*_{p}}{\longrightarrow}0
\end{displaymath}
which leads to the following isomorpism 
\begin{equation}
H^{\bullet}_X(f)\cong H^{\bullet}_{dR}(M)\oplus H^{\bullet-1}_{dR}(M^{\prime}).
\label{IIxx}
\end{equation}

In the following we present a link between our cohomology and a generalization of a classical $1$-differentiable  cohomology associated to an one form of Lichnerowicz type.

For a smooth map $f:M^{\prime}\rightarrow M$ we can  consider the set $\Omega^{\prime p}(f)=\Omega^p(M^{\prime})\oplus\Omega^{p-1}(M)$ and define the following operator
\begin{equation}
d_{\eta,f}:\Omega^{\prime p}(f)\rightarrow\Omega^{\prime p+1}(f)\,,\,d_{\eta,f}(\varphi,\psi)=(d\varphi-f^*(d\eta\wedge\psi),-d\psi).
\label{f12}
\end{equation}

An easy calculation leads to $d_{\eta,f}^2=(0,0)$. We denote the cohomology groups of the complex $(\Omega^{\prime\bullet}(f),d_{\eta,f})$ by $H^{\bullet}_{\eta}(f)$ and we call it the \textit{relative $1$-differentiable cohomology associated to the $1$-form $\eta$}.  We also notice that if $f$ is diffeomorphism then $\Omega^{\prime\bullet}(f)=\Omega^{\bullet}(f^{-1})$.

As in the case of $d_{X,f}$-cohomology, we  can consider the mappings $\mu:\Omega^{p}(M^{\prime})\rightarrow\Omega^{\prime p}(f)$, $\mu(\varphi)=(\varphi,0)$ and $\nu:\Omega^{\prime p}(f)\rightarrow\Omega^{p-1}(M)$, $\nu(\varphi,\psi)=\psi$ for all $\varphi\in\Omega^p(M^{\prime})$ and $\psi\in\Omega^{p-1}(M)$, respectively. Then, we have the following result which relates $H_{\eta}^{\bullet}(f)$ with the de Rham cohomologies of $M$ and $M^{\prime}$, respectively.
\begin{proposition}
\label{p3.1}
Let $f:M^{\prime}\rightarrow M$ be a smooth map between two  manifolds. Then:
\begin{enumerate}
\item[(i)] The mappings $\mu$ and $\nu$ induce an exact sequence of complexes
\begin{displaymath}
0\longrightarrow(\Omega^{\bullet}(M^{\prime}), d)\stackrel{\mu}{\longrightarrow}(\Omega^{\prime\bullet}(f),d_{\eta,f})\stackrel{\nu}{\longrightarrow}(\Omega^{\bullet-1}(M),-d)\longrightarrow0.
\end{displaymath}
\item[(ii)] This exact sequence induces a long exact cohomology sequence
\begin{equation}
\ldots\longrightarrow H^{p}_{dR}(M^{\prime})\stackrel{\mu^*}{\longrightarrow}H_{\eta}^p(f)\stackrel{\nu^*}{\longrightarrow}H^{p-1}_{dR}(M)\stackrel{\Delta^*_{p-1}}{\longrightarrow}H^{p+1}_{dR}(M^{\prime})\longrightarrow\ldots,
\label{f13}
\end{equation}
where the connecting homomorphism $\Delta^*_{p-1}$ is defined by
\begin{equation}
\Delta^*_{p-1}[\psi]=[-f^*(d\eta\wedge\psi)],\,\,{\rm for\,\,any}\,\,[\psi]\in H^{p-1}_{dR}(M).
\label{f14}
\end{equation}
\end{enumerate}
\end{proposition}
Since for any closed form $\psi\in\Omega^{p-1}(M)$ we have $f^*(d\eta\wedge\psi)=df^*(\eta\wedge\psi)$ the connecting homomorphism $\Delta^*_{p}$ vanishes and so we obtain the isomorphism:
\begin{equation}
H^{\bullet}_{\eta}(f)\cong H^{\bullet}_{dR}(M^{\prime})\oplus H^{\bullet-1}_{dR}(M).
\label{f15}
\end{equation}
Now, taking into account the isomorphisms \eqref{IIxx} and \eqref{f15} we obtain:
\begin{theorem}
If $f:M^{\prime}\rightarrow M$ is a difeomorphism then the following isomorphism hold:
\begin{equation}
\label{f16}
H_X^{\bullet}(f)\cong H_{\eta}^{\bullet}(f^{-1}).
\end{equation}
\end{theorem}

Finally, we notice that the $1$-differentiable cohomology induced by a holomorphic vector field can be generalized in the complex case with respect to a holomorphic map $f:M^\prime\rightarrow M$. More exactly, we can  consider a more general definition of  $\overline{\partial}_X$ on the set $\Omega^{p,q}(f)=\Omega^{p,q}(M)\oplus\Omega^{p,q-1}(M^{\prime})$ by setting
\begin{equation}
\overline{\partial}_{X,f}:\Omega^{p,q}(f)\rightarrow\Omega^{p,q+1}(f)\,,\,\overline{\partial}_{X,f}(\varphi,\psi)=(\overline{\partial}\varphi,\mathcal{L}_{X}f^*\varphi-\overline{\partial}\psi)\,,\,\,X\in\mathcal{X}(M^\prime).
\label{c3}
\end{equation}
An easy calculation shows that $\overline{\partial}_{X,f}^2=(0,0)$. Denote the cohomology groups of the complex $(\Omega^{p,\bullet}(f),\overline{\partial}_{X,f})$ by $H^{p,\bullet}_{X}(f)$ and we call it the \textit{relative $p$-th Dolbeault cohomology induced by the holomorphic vector field $X$} of $M^\prime$. As well as we seen if $M^{\prime}=M$ and $f=Id|_M$ then $\overline{\partial}_X=\overline{\partial}_{X,Id|_M}$ and $\widetilde{H}_X^{p,\bullet}(M)\equiv H^{p,\bullet}_{X}(Id|_M)$. We also notice that $\overline{\partial}_{X,f}$ is an antiderivation with respect to the exterior product from \eqref{xy1} extended to $\widetilde{\Omega}^{p,q}(f)$.

Using the same technique as above we  can consider the mappings $\overline{\alpha}:\Omega^{p,q-1}(M^{\prime})\rightarrow\Omega^{p,q}(f)$, $\overline{\alpha}(\psi)=(0,\psi)$ and $\overline{\beta}:\Omega^{p,q}(f)\rightarrow\Omega^{p,q}(M)$, $\overline{\beta}(\varphi,\psi)=\varphi$ for all $\varphi\in\Omega^{p,q}(M)$ and $\psi\in\Omega^{p,q-1}(M^{\prime})$, respectively. Then,  we have the following exact sequence of complexes:
\begin{displaymath}
0\longrightarrow(\Omega^{p,\bullet-1}(M^{\prime}), \overline{\partial})\stackrel{\overline{\alpha}}{\longrightarrow}(\Omega^{p,\bullet}(f),\overline{\partial}_{X,f})\stackrel{\overline{\beta}}{\longrightarrow}(\Omega^{p,\bullet}(M),\overline{\partial})\longrightarrow0.
\end{displaymath}
which induces a long exact cohomology sequence
\begin{equation}
\ldots\longrightarrow H^{p,q-1}(M^{\prime})\stackrel{\overline{\alpha}^*}{\longrightarrow}H_{X}^{p,q}(f)\stackrel{\overline{\beta}^*}{\longrightarrow}H^{p,q}(M)\stackrel{\overline{\delta}^*_{p,q}}{\longrightarrow}H^{p,q}(M^{\prime})\longrightarrow\ldots,
\label{c43}
\end{equation}
where the connecting homomorphism $\overline{\delta}^*_{p,q}$ is defined by
\begin{equation}
\overline{\delta}^*_{p,q}([\varphi])=[\mathcal{L}_{X}f^*\varphi],\,\,{\rm for\,\,any}\,\,[\varphi]\in H^{p,q}(M).
\label{c5}
\end{equation}
Taking into account that $X\in\mathcal{X}(M^\prime)$ is holomorphic and $f^*$ preserves the type of complex $(p,q)$ forms, a similar calculation as in the proof of Theorem \ref{t2.1} leads to
\begin{corollary}
\label{c2.2}
If $(M,g)$ is a compact K\"{a}hler manifold then we have the following isomorphism:
\begin{equation}
H^{p,\bullet}_X(f)\cong H^{p,\bullet}(M)\oplus H^{p,\bullet-1}(M^{\prime}).
\label{aa}
\end{equation}
\end{corollary}

\section{An application to harmonicity of $1$-differentiable forms}
\setcounter{equation}{0}

In this section we discuss some aspects concerning harmonicity of $1$-differentiable forms with respect to our operator in a particular case.

Let us consider a $n$-dimensional compact smooth manifold $M$, a closed $1$-form $\omega$ on $M$ and a Riemannian metric $g$ on $M$ such that $\omega$ is parallel with respect to $g$. We consider the vector field $U\in\mathcal{X}(M)$ characterized by the condition
\begin{equation}
\label{IV1}
\omega(X)=g(X,U)\,,\,\,\forall\,X\in\mathcal{X}(M).
\end{equation}
Also, we consider the codifferential operator $\delta$ given by
\begin{equation}
\label{IV2}
\delta\varphi=(-1)^{np+n+1}(\star\circ d\circ\star)\varphi\,,\,\forall\,\varphi\in\Omega^p(M),
\end{equation}
where $\star$ is the Hodge star operator, see for instance \cite{Go,Va}. 

Using that $\omega$ is parallel and that $U$ is Killing, we have the well known relations (see \cite{L-L-M2})
\begin{equation}
\label{IV3}
\mathcal{L}_U=-\delta\circ e(\omega)-e(\omega)\circ\delta\,,\,\delta\circ\mathcal{L}_U=\mathcal{L}_U\circ\delta,
\end{equation}
where $e(\omega)\varphi=\omega\wedge\varphi$ for every $\varphi\in\Omega^\bullet(M)$.

Now we define the following codifferential induced by the vector field $U$ associated to the operator $\widetilde{d}_U$ on the space $\widetilde{\Omega}^p(M)$ of $1$-differentiable $p$-forms:
\begin{equation}
\label{IV4}
\widetilde{\delta}_U:\widetilde{\Omega}^{p}(M)\rightarrow\widetilde{\Omega}^{p-1}(M)\,,\,\widetilde{\delta}_U(\varphi,\psi)=(\delta\varphi+\mathcal{L}_U\psi,-\delta\psi).
\end{equation}
By straightforward calculus, using the second relation of \eqref{IV3}, we obtain $\widetilde{\delta}_U^2=(0,0)$.

Using the natural scalar product $\langle,\rangle$ on the space $\Omega^\bullet(M)$ given by
\begin{equation}
\label{IV5}
\langle,\rangle:\Omega^p(M)\times\Omega^p(M)\rightarrow \mathbb{R}\,,\,(\varphi,\varphi^\prime)\mapsto\langle\varphi,\varphi^\prime\rangle=\int_M\varphi\wedge\star\varphi^\prime,
\end{equation}
we also define a scalar product $\langle\langle,\rangle\rangle$ on $\widetilde{\Omega}^\bullet(M)$ by
\begin{equation}
\langle\langle(\varphi,\psi),(\varphi^{\prime},\psi^{\prime})\rangle\rangle=\langle\varphi,\varphi^{\prime}\rangle+\langle\psi,\psi^{\prime}\rangle,\,\,\forall\,\,(\varphi,\psi),\,(\varphi^{\prime},\psi^{\prime})\in\widetilde{\Omega}^p(M).
\label{IV6}
\end{equation}
We have

\begin{proposition} 
The operator $\widetilde{d}_U$ is the $\langle\langle,\rangle\rangle$--adjoint of $\widetilde{\delta}_U$ and conversely. 
\end{proposition}
\begin{proof}
It is sufficient to prove that $\langle\langle\widetilde{d}_U(\varphi,\psi),(\varphi^{\prime},\psi^{\prime})\rangle\rangle=\langle\langle(\varphi,\psi),\widetilde{\delta}_U(\varphi^{\prime},\psi^{\prime})\rangle\rangle$, for every $(\varphi,\psi)\in\widetilde{\Omega}^{p-1}(M)$ and $(\varphi^{\prime},\psi^{\prime})\in\widetilde{\Omega}^{p}(M)$. By \eqref{IV6} we have
\begin{equation}
\langle\langle\widetilde{d}_U(\varphi,\psi),(\varphi^{\prime},\psi^{\prime})\rangle\rangle=\langle d\varphi,\varphi^{\prime}\rangle+\langle \mathcal{L}_U\varphi,\psi^\prime\rangle-\langle d\psi,\psi^{\prime}\rangle
\label{IV7}
\end{equation}
and 
\begin{equation}
\langle\langle(\varphi,\psi),\widetilde{\delta}_U(\varphi^{\prime},\psi^{\prime})\rangle\rangle=\langle \varphi,\delta\varphi^{\prime}\rangle+\langle \varphi,\mathcal{L}_U\psi^\prime\rangle-\langle \psi,\delta\psi^{\prime}\rangle.
\label{IV8}
\end{equation}
Now, the result follows taking into acount that $d$ is the $\langle,\rangle$--adjoint of $\delta$ and $\mathcal{L}_U$ is $\langle,\rangle$--self-adjoint.
\end{proof}

Let us define now a  $\widetilde{\Delta}_U$--Laplacian associated to the vector field $U$ on the space $\widetilde{\Omega}^p(M)$ by
$\widetilde{\Delta}=\widetilde{d}_U\widetilde{\delta}_U+\widetilde{\delta}_U\widetilde{d}_U$. Then by direct calculations, we obtain 
\begin{equation}
\widetilde{\Delta}_U(\varphi,\psi)=(\Delta\varphi+\mathcal{L}_U^2\varphi,\Delta\psi+\mathcal{L}_U^2\psi),\,\forall\,(\varphi,\psi)\in\widetilde{\Omega}^p(M).
\label{IV9}
\end{equation}
Hence, the $1$-differentiable form $(\varphi,\psi)\in\widetilde{\Omega}^p(M)$ is \textit{$\widetilde{\Delta}_U$--harmonic}, that is $\widetilde{\Delta}_U(\varphi,\psi)=(0,0)$ iff
\begin{equation}
\label{IV10}
\Delta\varphi+\mathcal{L}_U^2\varphi=0\,\,{\rm and}\,\,\Delta\psi+\mathcal{L}_U^2\psi=0.
\end{equation}
\begin{proposition}
The $1$-differentiable form $(\varphi,\psi)$ is $\widetilde{\Delta}_U$--harmonic iff $\widetilde{d}_U(\varphi,\psi)=(0,0)$ and $\widetilde{\delta}_U(\varphi,\psi)=(0,0)$.
\end{proposition}
\begin{proof}
It follows directly from 
\begin{eqnarray*}
\langle\langle\widetilde{\Delta}_U(\varphi,\psi),(\varphi,\psi)\rangle\rangle
&=&\langle\langle\widetilde{d}_U\widetilde{\delta}_U(\varphi,\psi)+\widetilde{\delta}_U\widetilde{d}_U(\varphi,\psi),(\varphi,\psi)\rangle\rangle\\
&=& \langle\langle\widetilde{\delta}_U(\varphi,\psi),\widetilde{\delta}_U(\varphi,\psi)\rangle\rangle+ \langle\langle\widetilde{d}_U(\varphi,\psi),\widetilde{d}_U(\varphi,\psi)\rangle\rangle.
\end{eqnarray*}
\end{proof}
Also, if we continue the calculus from the proof of above proposition, we obtain
\begin{equation}
\label{IV11}
\langle\langle\widetilde{\Delta}_U(\varphi,\psi),(\varphi,\psi)\rangle\rangle=\langle\langle(\varphi,\psi),\widetilde{\Delta}_U(\varphi,\psi)\rangle\rangle
\end{equation}
which says that the $\widetilde{\Delta}_U$-Laplacian is $\langle\langle,\rangle\rangle$-self-adjoint.

\begin{example}
Let us consider a smooth function $f\in C^\infty(M)$ with the property that $\mathcal{L}_Uf=kf$, $k\in\mathbb{R}$. Then, the $1$-differentiable $0$-form $(f,0)\in\widetilde{\Omega}^0(M)$ is $\widetilde{\Delta}_U$-harmonic iff the function $f$ satisfies the Klein-Gordon type equation
\begin{equation}
\label{IV15}
\Delta f+k^2f=0,
\end{equation}
with mass term $m^2=k^2$. 
\end{example}

A first interpretation of formula \eqref{IV9} in terms of usual harmonicity can be obtained  as follows. As $U$ is a Killing vector field, it is well known that in this case $\mathcal{L}_U\varphi=0$ for every harmonic form $\varphi\in\Omega^\bullet(M)$. Thus, we have
\begin{proposition}
If $\varphi\in\Omega^p(M)$ and $\psi\in\Omega^{p-1}(M)$ are both harmonic then the $1$-differentiable $p$-form $(\varphi,\psi)$ is $\widetilde{\Delta}_U$-harmonic.
\end{proposition}

\begin{remark}
We notice that a similar Laplacian can be defined for compact Sasakian manifolds (or more generally for $K$-contact manifolds) with respect to the Reeb vector field $\xi$ which is a Killing vector field.
\end{remark}

Another interpretation of formula \eqref{IV9} can be obtained involving the Lichnerowicz-Laplacian. More exactly, as $\omega$ is a closed $1$-form on $M$ there is an associated Lichnerowicz differential, see \cite{G-L, L-L-M1, L}, $d_\omega:\Omega^p(M)\rightarrow\Omega^{p+1}(M)$ defined by 
\begin{equation}
\label{IV12}
d_\omega\varphi=d\varphi+e(\omega)\varphi\,,\,\forall\,\varphi\in\Omega^p(M).
\end{equation}
By direct calculus we easily obtain that $d_\omega^2=0$ and the cohomology $H^\bullet_\omega(M)$ of the resulting complex $(\Omega^\bullet(M),d_\omega)$ is called the Lichnerowicz cohomology of $(M,\omega)$. Also, a codifferential operator $\delta_\omega:\Omega^p(M)\rightarrow\Omega^{p-1}(M)$ associated to $\omega$ is given by (see \cite{G-L})
\begin{equation}
\label{IV13}
\delta_\omega=\delta+\imath_U.
\end{equation}
Then, the Lichnerowicz-Laplacian is defined as $\Delta_\omega=d_\omega\delta_\omega+\delta_\omega d_\omega$ and by direct calculus, taking into account the first relation of \eqref{IV3} we get
\begin{equation}
\label{IV14}
\Delta_\omega\varphi=\Delta\varphi+\varphi\,,\,\forall\,\varphi\in\Omega^p(M).
\end{equation}
Also, using an associated Hodge decomposition (with respect to the operators $\Delta_\omega$, $d_\omega$ and $\delta_\omega$, see \cite{G-L}), it is proved, see \cite{L-L-M2}, that for every $\Delta_\omega$-harmonic form $\varphi$ we have $\mathcal{L}_U\varphi=-\varphi$. 

Hence, we have the following result
\begin{proposition}
If $\varphi\in\Omega^p(M)$ and $\psi\in\Omega^{p-1}(M)$ are both $\Delta_\omega$-harmonic then the $1$-differentiable $p$-form $(\varphi,\psi)$ is $\widetilde{\Delta}_U$-harmonic.
\end{proposition}

\section*{Acknowledgement} The second author is supported by the Sectorial Operational Program Human Resources Development (SOP HRD), financed from the European Social Fund and by the Romanian Government under the Project number POSDRU/159/1.5/S/134378 .

\noindent
Mircea Crasmareanu \\
Faculty of Mathematics \newline
University "Al. I. Cuza" \newline
Address: Ia\c si, 700506, Bd. Carol I, no. 11, Rom\^ania \newline
email: \textit{mcrasm@uaic.ro}

\medskip
\noindent
Cristian Ida\\
Department of Mathematics and Computer Science\\
University Transilvania of Bra\c{s}ov\\
Address: Bra\c{s}ov 500091, Str. Iuliu Maniu 50, Rom\^{a}nia\\
email: \textit{cristian.ida@unitbv.ro}

\medskip
\noindent 
Paul Popescu\\
Department of Applied Mathematics\\
 University of Craiova\\
Address: Craiova, 200585,  Str. Al. Cuza, No. 13,  Rom\^{a}nia\\
 email:\textit{paul$_{-}$p$_{-}$popescu@yahoo.com}


\begin{thebibliography}{11}
\bibitem{B-T} R. Bott, L. W. Tu, \textit{Differential Forms in Algebraic Topology}. Graduate Text in Math., \textbf{82}, Springer-Verlag, Berlin, 1982. 
\bibitem{c:ds} A. Cannas da Silva, {\it Lectures on symplectic geometry}, Lecture Notes in Mathematics, 1764, Springer-Verlag, Berlin, 2001. 
\bibitem{F-L-S} M. Flato, A. Lichnerowicz, D. Sternheimer, \textit{D\'{e}formations $1$-diff\'{e}rentiables des alg\`{e}bres de Lie attach\'{e}es \`{a} une vari\'{e}t\'{e} symplectique ou de contact}. Compositio Math., \textbf{31} no. 1 (1975), 47-82.
\bibitem{Go} S. I. Goldberg, {\em Curvature and Homology}. Revised Edition. ISBN 0-486-40207-X, Dover Publication, Inc. Mineola, New-York 1998.
\bibitem{G-L} F. Gu\'{e}dira, A. Lichnerowicz, {\em G\'{e}ometrie des alg\'{e}bres de Lie locales de Kirillov}. J. Math. Pures et Appl., \textbf{63} (1984), 407--484.

\bibitem{I-M} C. Ida, S. Merche\c{s}an, {\em On harmonic and C--harmonic $1$--differentiable forms on Sasakian manifolds}.  Mediterr. J. Math. \textbf{11} (2014), 155-171.
\bibitem{Ko}  S. Kobayashi,  {\em Differential geometry of complex vector bundles}, Publ. of the Math. Soc. of Japan, 15. Kan\^{o} Memorial Lectures, 5. \textit{Princeton Univ. Press, Princeton, NJ; Iwanami Shoten, Tokyo}, 1987.
\bibitem{L-L-M1} M. de Le\'{o}n, J.C. Marrero, E. Padr\'{o}n, {\em Lichnerowicz-Jacobi cohomology}. J. Phys. A:
Math. Gen., \textbf{30} (1997), 6029-6055.
\bibitem{L-L-M2} M. de Le\'{o}n, B. L\'{o}pez, J. C. Marrero, E. Padr\'{o}n, {\it Lichnerowicz-Jacobi cohomology and homology of Jacobi manifolds: modular class and duality}. Available to arXiv: math/9910079v1 [math.DG] 1999.
\bibitem{L-L-M} M. de Le\'{o}n, B. L\'{o}pez, J. C. Marrero, E. Padr\'{o}n, {\em On the computation of the Lichnerowicz-Jacobi cohomology}, J. Geom. Phys. \textbf{44} (2003), 507-522.
\bibitem{L1} A. Lichnerowicz., C. R. Acad. Sc. Paris. t. 277, A, 1973, 215-219: {\em Cohomologie $1$-diferentiable des alg\'{e}bres de Lie attach\'{e}es a une vari\'{e}t\'{e} symplectique ou de contact}, J. Math. Pures et Appl., \textbf{53} (1974), 459-484.

\bibitem{L} A. Lichnerowicz, {\em Les vari\'{e}t\'{e}s de Poisson et leurs alg\'{e}bres de Lie associ\'{e}es}, J. Differential Geom. \textbf{12} (2) (1977), 253-300.

\bibitem{Mich1} P. W. Michor, {\it Topics in differential geometry}. Graduate Studies in Mathematics, \textbf{93}, AMS, Providence, RI, 2008.

\bibitem{Mo-Ko} J. Morrow, K.  Kodaira,  {\em Complex Manifolds}, AMS Chelsea Publ., New York,  1971.

\bibitem{N-R1} P. Nurowski, D. C. Robinson, {\em Generalized exterior forms, geometry and space-time}. Class. Quantum Grav. \textbf{18} (2001) L81--L86.

\bibitem{N-R2} P. Nurowski, D. C. Robinson, {\em Generalized forms and their applications}. Class. Quantum Grav. \textbf{19} (2002) 2425--2436.
\bibitem{R} D. C. Robinson, {\em Generalized forms, Chern-Simons and Einstein-Yang-Mills theory}. Class. Quantum Grav. \textbf{26} (2009), 075019. 

\bibitem{Va} I. Vaisman, \textit{Cohomology and differential forms}. Pure and Applied Mathematics, {\bf 21}, M. Dekker Inc., New York, 1973.
\end{thebibliography}
\end{document}